\documentclass[11pt, a4paper]{amsart}
\usepackage{latexsym,mathrsfs,bm}
\usepackage{graphicx,color,url}
\usepackage{amssymb,mathrsfs}
\usepackage{bm}
\usepackage{hyperref}

\usepackage[english]{babel}
\usepackage{paralist}
\usepackage{amsthm,amsfonts,amssymb,amsmath}
\usepackage{mathtools}
\usepackage[utf8]{inputenc}

\newtheorem*{theo*}{Theorem}

\newtheorem{theorem}{Theorem}

\newtheorem{lemma}{Lemma}

\newtheorem*{rem*}{Remark}

\let\originalleft\left
\let\originalright\right
\renewcommand{\left}{\mathopen{}\mathclose\bgroup\originalleft}
\renewcommand{\right}{\aftergroup\egroup\originalright}

\definecolor{pink}{rgb}{1,.2,.6}
\definecolor{orange}{rgb}{0.7,0.3,0}
\definecolor{blue}{rgb}{.2,.6,.75}
\definecolor{green}{rgb}{.4,.7,.4}

\newcommand{\suppress}[1]{}

\numberwithin{equation}{section}

\title{On the trace of the integers of a number field} 

\author[F.~Battistoni]{Francesco Battistoni}
\address{Dipartimento di Matematica, Università degli Studi di Milano, Via Saldini 50, 20133 Milano, Italy}
\email{francesco.battistoni@unimi.it}

\author[T.~Za\"{\i}mi]{Toufik Za\"{\i}mi}
\address{Department of Mathematics and Statistics, College of Science, Imam Mohammad Ibn Saud Islamic University (IMSIU), P. O. Box 90950, Riyadh 11623, Saudi Arabia}
\email{tmzaemi@imamu.edu.sa}

\makeatletter
\@namedef{subjclassname@2020}{%
  \textup{2020} Mathematics Subject Classification}
\makeatother

\begin{document}

\begin{abstract}
    Let $Tr$ denote the trace $%
    \mathbb{Z}$-module homomorphism defined on the ring $\mathcal{O}%
_{L} $ of the integers of a number field $L.$ We show that 
$Tr(\mathcal{O}_{L})\varsubsetneq \mathbb{Z}$ if and only if there
is a prime \ factor $p$ of the degree of $L$ such that
if $\wp _{1}^{e_{1}}...\wp _{s}^{e_{s}}$ is the prime factorization
of the ideal $p\mathcal{O}_{L}$ in $\mathcal{O}_{L},$
then $p$ divides all powers $e_{1},...,e_{s}.$ Also, we
prove that the equality $Tr(\mathcal{O}_{L})=\mathbb{Z}$ holds
when $L$ is the compositum of certain number fields. 
\end{abstract}



\maketitle

\section{ Introduction}

For a finite extension $L$ of a number field $K,$ the trace homomorphism is
the mapping $T_{L/K}:L\rightarrow K,$ defined for all $\alpha \in L$ by 

\[
\text{ \ \ \ \ }T_{L/K}(\alpha )=\sigma _{1}(\alpha )+\cdot \cdot \cdot
+\sigma _{d}(\alpha ), 
\]%

where $\sigma _{1},...,\sigma _{d}$ denote the $d=[L:K]$ embeddings of $L$
in the field of complex numbers $\mathbb{C}$ which fix $K$ pointwise. Then,
the function $T_{L/K}$ is $K$-linear, i.e., $T_{L/K}(k\alpha +\alpha
^{\prime })=kT_{L/K}(\alpha )+T_{L/K}(\alpha ^{\prime }),$ $\ \forall
(k,\alpha ,\alpha ^{\prime })\in K\times L\times L,$ and satisfies $%
T_{L/K}(L)=K,$ since $T_{L/K}(k/d)=k$ \ for all $k\in K.$ Also, if $M$ is a
finite extension of $L,$ then 

\begin{equation}\label{eq1}\tag{1}
T_{M/K}(\beta )=T_{L/K}(T_{M/L}(\beta )),\text{ \ }\forall \beta \in M. 
\end{equation}%

Let $\mathcal{O}_{L}$ be the ring of the integers of $L.$ Then, $\mathcal{O}%
_{L}$ is an $\mathcal{O}_{K}$-module, and the restriction, say $T_{\mathcal{O%
}_{L}/\mathcal{O}_{K}},$ of $T_{L/K}$ to $\mathcal{O}_{L}$ is an $\mathcal{O}%
_{K}$-module homomorphism whose range is a non-zero ideal of $\mathcal{O}%
_{K}.$ In particular, if $K$ is the field of the rational numbers $\mathbb{Q}%
,$ then there is a unique natural number $t_{L}$ such that 
\[
T_{\mathcal{O}_{L}/\mathbb{Z}}(\mathcal{O}_{L})=t_{L}\mathbb{Z}, 
\]%
where $\mathbb{Z=}\mathcal{O}_{\mathbb{Q}}$ is the ring of the rational
integers. Moreover, the integer $t_{L}$ divides $T_{\mathcal{O}_{L}/\mathbb{Z%
}}(1)=\deg (L)=d,$ and the operator $T_{\mathcal{O}_{L}/\mathbb{Z}}$ is
surjective if and only if $t_{L}=1.$

In \cite{battistoniTrace} the first named author gave conditions on the field $L$ under which
the equality $t_{L}=1$ holds. For example \cite[Theorem 1]{battistoniTrace} (see also \cite[Corollary 5]{narkiewicz2013elementary}) says that $T_{\mathcal{O}_{L}/\mathbb{Z}}$ is surjective
when $L$ is tame. Recall that the field $L$ is said to be tame (over $%
\mathbb{Q}$) whenever each prime number $p$ is tamely ramified in $L,$ i.
e., if $\wp _{1}^{e_{1}}...\wp _{s}^{e_{s}}$ is the prime factorization of
the ideal $p\mathcal{O}_{L}$ in $\mathcal{O}_{L},$ then $p$ does not divide
any $e_{i},$ $\forall $ $i\in \{1,...,s\}.$ A wild number field is a number
field which is not tame, and a prime number which is not tamely ramified in $%
L$ is said to be wildly ramified in $L.$

Using a related result of Dedekind, stated in the next section, it is easy
to see that if $p$ is wildly ramified in $L,$ then $p^{p}$ divides the
(absolute) discriminant $\text{disc}(L)$ of $L.$ Therefore, if $\text{%
disc}(L)$ is squarefree, then $L$ is tame and so $t_{L}=1.$

Another surjectivity criterion, stated in \cite{battistoniTrace} for certain wild fields, says
that if $\text{disc}(L)$ is not divisible by the squares of the prime
factors of $d,$ then $T_{\mathcal{O}_{L}/\mathbb{Z}}(\mathcal{O}_{L})=%
\mathbb{Z}.$ In fact, we can obtain more: writing $\text{disc}(L)$ as the
determinant of the matrix $[T_{\mathcal{O}_{L}/\mathbb{Z}}(\mathcal{\omega }%
_{i}\mathcal{\omega }_{j})]_{1\leq i,j\leq d},$ where $\{\mathcal{\omega }%
_{1},...,\mathcal{\omega }_{d}\}$ is a base of the free $\mathbb{Z}$-module $%
\mathcal{O}_{L},$ we see that $t_{L}^{d}$ divides $\text{disc}(L),$ and
so if $T_{\mathcal{O}_{L}/\mathbb{Z}}$ is non-surjective, then there is a
prime number $p,$ dividing $t_{L}$ (and so $d),$ such that $p^{d}$ is a
factor of $\text{disc}(L).$ The aim of the present note is to continue
the investigation of this question.

In order to state our first result, let us recall some related notions (for
more details see \cite[Section 4.2]{narkiewicz2013elementary}). The codifferent of the ring $\mathcal{O}%
_{L}$ (over $\mathbb{Q)}$ is the set \ 

\[
\widehat{\mathcal{O}_{L}}:=\{\alpha \in L\mid T_{L/\mathbb{Q}}(\alpha 
\mathcal{O}_{L}\mathcal{)}\subset \mathbb{Z}\}. 
\]%

It is clear that $\mathcal{O}_{L}\subset \widehat{\mathcal{O}_{L}},$ and $%
\widehat{\mathcal{O}_{L}}$ is an $\mathcal{O}_{L}$-submodule of $L.$ Let $\{%
\widehat{\mathcal{\omega }_{1}},...,\widehat{\mathcal{\omega }_{d}}\}$ be
the dual base of $\{\mathcal{\omega }_{1},...,\mathcal{\omega }_{d}\},$ that
is, the $\mathbb{Q}$-base of $L$ defined by the equalities $T_{L/\mathbb{Q}}(%
\mathcal{\omega }_{i}\widehat{\mathcal{\omega }_{j}})=1$ if $i=j,$ and $T_{L/%
\mathbb{Q}}(\mathcal{\omega }_{i}\widehat{\mathcal{\omega }_{j}})=0$
otherwise. Then, $\{\widehat{\mathcal{\omega }_{1}},...,\widehat{\mathcal{%
\omega }_{d}}\}$ is a $\mathbb{Z}$-base of $\widehat{\mathcal{O}_{L}},$ and
so there is $b\in \mathbb{Z\diagdown }\{0\}\subset \mathcal{O}_{L}\mathbb{%
\diagdown }\{0\}$ such that $b\widehat{\mathcal{O}_{L}}\subset \mathcal{O}%
_{L}.$ Hence, $\widehat{\mathcal{O}_{L}}$ is fractional ideal in $L,$ and
its inverse

\[
\mathcal{D}_{L}:=\widehat{\mathcal{O}_{L}}^{-1}=\{\alpha \in L\mid \alpha 
\widehat{\mathcal{O}_{L}}\subset \mathcal{O}_{L}\}, 
\]%
namely the different of $L,$ is also a fractional ideal in $L.$ In fact, $%
\mathcal{D}_{L}$ is an ideal of $\mathcal{O}_{L},$ because if $\alpha \in 
\mathcal{D}_{L},$ then $\alpha \mathcal{O}_{L}\subset \alpha \widehat{%
\mathcal{O}_{L}}\subset \mathcal{O}_{L}$ and so $\alpha \in \mathcal{O}_{L}.$

\begin{theorem}\label{theorem1}
Let $L$ be a number field of degree $%
d, $\textit{\ }$\mathcal{D}_{L}$ the different of $L,$ and 
$t_{L}\mathbb{Z}$ the range of the trace operator on the ring $%
\mathcal{O}_{L}$ of the integers of $L.$ Then, the
following assertions are equivalent:

\begin{itemize}

\item[(i)] $t_{L}\geq 2;$

\item[(ii)] there is a prime factor $p$ of $d$ such that
if $\wp _{1}^{e_{1}}...\wp _{s}^{e_{s}}$ is the prime factorization
of the ideal $p\mathcal{O}_{L}$ in $\mathcal{O}_{L},$
then $p$ divides all powers $e_{1},...,e_{s};$

\item[(iii)] there is a prime number $p$ such that if $\wp
_{1}^{e_{1}}...\wp _{s}^{e_{s}}$ is the prime factorization of the
ideal $p\mathcal{O}_{L}$ in $\mathcal{O}_{L},$ then $p$%
 divides all integers $e_{1},...,e_{s}.$
\end{itemize}
\end{theorem}

For example, if the degree $d$ of $L$ has a prime factor, say again $p,$
which is totally ramified in $L,$ i. e., there is a prime ideal $\wp $ of 
\textit{\ }$\mathcal{O}_{L}$ such that $\wp ^{d}$ is the prime factorization
of $p\mathcal{O}_{L}$\textit{\ }in\textit{\ }$\mathcal{O}_{L},$ then Theorem \ref{theorem1}.(ii) is true and so $t_{L}\geq 2.$ This holds if and only if $L=\mathbb{Q}%
(\alpha ),$ where $\alpha $ is a root of a monic integer polynomial which is
Eisenstein at a prime factor of its degree. 

In view of Theorem \ref{theorem1} and Lemma \ref{lemma2} below, the following theorem is immediate.

\begin{theorem}\label{theorem2}

 With the notation above, consider the following
three propositions: 

\begin{itemize}
    \item[(A)] $t_{L}\geq 2;$
    \item[(B)] there is a prime factor $p$ of $d$ such
that $p^{d}$ divides $\text{disc}(L);$
\item[(C)] $L$ is wild. 
\end{itemize}

Then, %
\[
(A)\Rightarrow (B)\Rightarrow (C). 
\]%

Moreover, %
\begin{equation}\label{eq2}
(B)\Rightarrow (A)  \tag{2}
\end{equation}%

when $d$ is prime or $d=4,$ and if $%
L $ is normal (over $\mathbb{Q}),$ then  
\[
\mathit{\ }(C)\Rightarrow (A). 
\]
\end{theorem}

It follows immediately that the three assertions in the statement of Theorem
2 are equivalent when $L$ is normal. For example if $L=\mathbb{Q}(\sqrt{m})$
is a quadratic field, where $m$ is a squarefree integer satisfying $m\equiv 1%
\mod 4$ (resp. $m\equiv 2,3\mod 4),$ then $\text{disc}(L)=m,$ $%
t_{L}=1$ and $L$ is tame (resp. $\text{disc}(L)=2^{2}m,$ $\ t_{L}\geq 2$
and $L$ is wild$).$ Also, if $d$ is prime, then $t_{L}\geq 2\Leftrightarrow
d^{d}$ divides $\text{disc}(L).$

However, when $L$ is not normal, the statement \eqref{eq2} already fails for $d=6$: in fact the number field $L=\mathbb{Q}(\alpha)$, where $\alpha^6+\alpha^4+5\alpha^2+1=0$, has degree $d=6$ and discriminant $-2^{10}\cdot 13^2$, but the ideal $2\mathcal{O}_L$ factorizes as $\wp_1^4 \wp_2$, so that the trace must be surjective by Theorem \ref{theorem1}. This can be explicitly verified considering the algebraic integer

\[
\beta := \frac{\alpha^5+\alpha^4+2\alpha^3+2\alpha^2-\alpha-1}{4}
\]
which has trace $-7$. These explicit computations may be verified using the computer algebra PARI/GP \cite{pari}.

It is then of interest to determine the non-normal fields of composite degree $d\geq 6$ for which \eqref{eq2} remains true.\\

Now, suppose that $M$ is a finite extension of $L,$ and let $\beta \in 
\mathcal{O}_{M}$ be such that $T_{\mathcal{O}_{M}/\mathbb{Z}}(\beta )=t_{M}.$
Then, \eqref{eq1} gives $T_{\mathcal{O}_{L}/\mathbb{Z}}(\alpha )=t_{M},$ where $%
\alpha :=T_{\mathcal{O}_{M}/\mathcal{O}_{L}}(\beta )\in \mathcal{O}_{L},$
and so $t_{L}$ divides $t_{M}.$ In particular, 
\[
t_{M}=1\Rightarrow t_{L}=1, 
\]%
and a related question arises immediately: Let $L_{1},...,L_{r}$ be numbers
fields, satisfying $t_{L_{1}}=$ $...$ $=t_{L_{r}}=1,$ and let $%
L:=L_{1}...L_{r}$ be their compositum, that is, the intersection of all
subfields of $\mathbb{C}$ containing $\cup _{1\leq i\leq r}L_{r}.$ Is the
equality $t_{L}=1$ always true? The following theorem is a partial answer to
this question.

\begin{theorem}\label{theorem3}We have the following results.
\medskip
\begin{itemize}
    \item[(I)] Let $L$ be the compositum of some number
fields $L_{1},...,L_{r},$ satisfying \\
$\gcd (\deg (L_{i}),\deg
(L_{j}))=1,$\textit{\ }$\forall $\textit{\ }$1\leq i\neq j\leq r,$ 
and $t_{L_{i}}=1$ for all $1\leq i\leq r.$ Then,  $\deg
(L)=\prod\limits_{1\leq i\leq r}\deg (L_{i})$ and $t_{L}=1.$

\item[(II)] If $L$\textit{\ is the compositum of the quadratic fields }$%
\mathbb{Q}(\sqrt{m_{1}}),...,\mathbb{Q}(\sqrt{m_{r}}),$ \textit{where the
squarefree integers }$m_{1},...,\mathbb{\ }m_{r}$ \textit{satisfy \ }$%
m_{1}\equiv $ $...$ $\equiv m_{r}\equiv 1\mod 4,$ \textit{then }$%
t_{L}=1.$

\item[(III)] If $L$\textit{ and }$K$\textit{ are normal fields satisfying }$t_L=t_K=1,$\textit{ then }$t_{KL}=1.$
\end{itemize}
\end{theorem}

Statement (III) is a generalization of statement (II): we will provide a constructive proof of (II) before presenting the one of (III).

As mentioned above, the first two theorems, proved in the third section, are
mainly based on a related result of Dedekind. This result together with some
properties of the different of a number field are presented in the next
section. Theorem \ref{theorem3} is proved in the last section.

\section{Some properties of the different of a number field}

The aim of this section is to collect some properties of the different $%
\mathcal{D}_{L}$ of a number field $L.$\ Most of these properties may be
found in \cite{marcus1977number, narkiewicz2013elementary}. With the notation of the introduction, we have%
\[
\mathcal{D}_{L}\subset \mathcal{O}_{L}\text{ }\mathcal{\subset }\text{ }%
\widehat{\mathcal{O}_{L}}, 
\]%
and $\mathcal{\omega }_{i}=\sum\limits_{j=1}^{d}T_{\mathcal{O}_{L}/\mathbb{Z%
}}(\mathcal{\omega }_{i}\mathcal{\omega }_{j})\widehat{\mathcal{\omega }_{j}}%
,$ $\forall i\in \{1,...,d=\deg (L)\}.$ Hence, the index $[\widehat{\mathcal{%
O}_{L}}:\mathcal{O}_{L}\mathcal{]}$ of the groups extension $\mathcal{O}%
_{L}\subset \widehat{\mathcal{O}_{L}}$ is equal to the absolute value of $%
\text{disc}(L),$ and so we obtain, by the identity $[\widehat{\mathcal{O}%
_{L}}:\mathcal{O}_{L}]=$ $[\mathcal{O}_{L}:\widehat{\mathcal{O}_{L}}^{-1}%
\mathcal{]},$ that the (absolute) norm of the ideal $\widehat{\mathcal{O}_{L}%
}^{-1}=\mathcal{D}_{L}$ is equal to $\left\vert \text{disc}(L)\right\vert
.$ The following lemma is a key tool in proving the first two theorems.

\begin{lemma}\label{lemma1}
Let $t_{L}\mathbb{Z}$ be the range of
the trace operator on the ring $\mathcal{O}_{L}$ of the integers
of $L.$ Then, %
\begin{equation}\label{eq3}
\mathcal{D}_{L}\subset t_{L}\mathcal{O}_{L}.  \tag{3}
\end{equation}%
Moreover, if $\mathcal{D}_{L}\subset n\mathcal{O}_{L}$ 
for some $n\in \mathbb{N},$ then $t_{L}\mathcal{O}_{L}\subset n%
\mathcal{O}_{L}.$
\end{lemma}

\begin{proof}
Since $\mathcal{D}_{L}\subset t_{L}\mathcal{O}%
_{L}\Leftrightarrow (t_{L}\mathcal{O}_{L})^{-1}\subset \mathcal{D}_{L}^{-1},$
to show the first assertion it suffices to prove that any element $\alpha $
of the fractional ideal $(t_{L}\mathcal{O}_{L})^{-1}=t_{L}^{-1}\mathcal{O}%
_{L}$ belongs to the codifferent $\mathcal{D}_{L}^{-1}=\widehat{\mathcal{O}%
_{L}}$ of $\mathcal{O}_{L}.$ In fact, this follows from the relations 
\[
\alpha \in t_{L}^{-1}\mathcal{O}_{L}\Leftrightarrow t_{L}\alpha \in \mathcal{%
O}_{L}\Rightarrow t_{L}\alpha \mathcal{O}_{L}\subset \mathcal{O}%
_{L}\Rightarrow T_{L/\mathbb{Q}}(t_{L}\alpha \mathcal{O}_{L})\subset t_{L}%
\mathbb{Z}, 
\]
and 
\[
t_{L}T_{L/\mathbb{Q}}(\alpha \mathcal{O}_{L})\subset t_{L}\mathbb{Z}%
\Leftrightarrow T_{L/\mathbb{Q}}(\alpha \mathcal{O}_{L}\mathcal{)}\subset 
\mathbb{Z}\Leftrightarrow \alpha \in \widehat{\mathcal{O}_{L}}. 
\]%
Similarly, if\textit{\ }$\mathcal{D}_{L}\subset n\mathcal{O}_{L},$ then $%
n^{-1}\mathcal{O}_{L}\subset \widehat{\mathcal{O}}_{L},$ and 
\[
n^{-1}t_{L}\mathbb{Z=}n^{-1}T_{L/\mathbb{Q}}(\mathcal{O}_{L})=T_{L/\mathbb{Q}%
}(n^{-1}\mathcal{O}_{L}\mathcal{)}\subset T_{L/\mathbb{Q}}(\widehat{\mathcal{%
O}_{L}})\subset \mathbb{Z}. 
\]%
Hence, $n^{-1}t_{L}\mathbb{Z}\subset \mathbb{Z},$ so that $n^{-1}t_{L}\in 
\mathbb{Z},$ $n$ divides $t_{L},$ and thus $t_{L}\mathcal{O}_{L}\subset n%
\mathcal{O}_{L}.$%
\end{proof}

The lemma below is an immediate corollary of a theorem of Dedekind
(see for instance \cite[Theorem 4.24]{narkiewicz2013elementary}).

\begin{lemma}\label{lemma2}
Let $\wp $ be a  prime ideal of 
$\mathcal{O}_{L}$ lying over a prime number $p$ with a
ramification index $e.$ If $p$ does not divide (
resp. $p$ divides)\textit{\ }$e,$ then the exact power
of $\wp $ in the prime factorization of the ideal $\mathcal{D}_{L}$
is $e-1$ (resp. is some natural number $n\in \lbrack
e,e-1+ev_{p}(e)],$ where $v_{p}(e)$ is the $p$
-adic valuation of $e$).
\end{lemma}

In fact, the above mentioned upper bound $e-1+ev_{p}(e)$ \ for the power $n$
of the prime ideal $\wp ,$ when $v_{p}(e)\geq 1,$ was conjectured by
Dedekind and proved by Hensel.

To make clear the proof of Theorem 2, we will also use the following simple
corollary of Lemma 2.

\begin{lemma}\label{lemma3}
\textit{Let }$p$\textit{\ be a prime number such that }$%
p^{n}$ \textit{divides} $\text{disc}(L)$ \textit{for some rational integer%
} $n\geq d.$ \textit{Then, }$p$\textit{\ is wildly ramified in }$L.$
\end{lemma}

\begin{proof}
    Let $\wp _{1}^{e_{1}}...\wp _{s}^{e_{s}}$ be the prime
factorization of the ideal $p\mathcal{O}_{L}$ in $\mathcal{O}_{L},$ and let $%
p^{f_{1}},...,p^{f_{s}}$ be the norms of the prime ideals $\wp _{1},...,\wp
_{s},$ respectively. Then, the norm $p^{d}$ of the principal ideal $p%
\mathcal{O}_{L}$ is also equal $p^{e_{1}f_{1}+\cdot \cdot \cdot
+e_{s}f_{s}}, $ and so $d=e_{1}f_{1}+\cdot \cdot \cdot +e_{s}f_{s}.$

Assume that $p$ is tamely ramified in $L,$ i. e., $p$ does not divide any $%
e_{i},$ $\forall i\in \{1,...,s\}.$ Then, Lemma 2\textbf{\ }gives that the
exact power of each prime $\wp _{i}$ in the prime factorization of $\mathcal{%
D}_{L}$ is $e_{i}-1,$ and so the exact power of $p$ in $\text{disc}(L)$
is $(e_{1}-1)f_{1}+\cdot \cdot \cdot +(e_{s}-1)f_{s}.$ This leads
immediately to a contradiction, because 
\[
(e_{1}-1)f_{1}+\cdot \cdot \cdot +(e_{s}-1)f_{s}=d-(f_{1}+\cdot \cdot \cdot
+f_{s})\leq d-s<d\leq n; 
\]%
thus $p$ is wildly ramified.%
\end{proof}

\section{Proofs of the first two theorems}

\textbf{Proof of Theorem 1. } To show the implication$\ (i)\Rightarrow (ii),$
suppose $t_{L}\geq 2$ and consider a prime factor $p$ of $t_{L}.$ Then, $p$
divides $d,$ since $t_{L}$ is a factor of $d.$ Also, the inclusion $t_{L}%
\mathcal{O}_{L}\subset p\mathcal{O}_{L}$ together with \eqref{eq3}, yields $%
\mathcal{D}_{L}\subset p\mathcal{O}_{L}.$ Let $\wp _{1}^{e_{1}}...\wp
_{s}^{e_{s}}$ be the prime factorization of the ideal $p\mathcal{O}_{L}$ in $%
\mathcal{O}_{L}.$ Then, 

\[
\mathcal{D}_{L}\subset \wp _{i}^{e_{i}},\text{ }\forall i\in \{1,...,s\}, 
\]%
i. e., each ideal $\wp _{i}^{e_{i}}$ divides $\mathcal{D}_{L},$ and so, by
Lemma \eqref{lemma2}, $p$ divides all powers $e_{1},...,e_{s}.$

It is also clear that the implication $(ii)\Rightarrow (iii)$ is trivially
true. In order to complete the proof suppose that $(iii)$ holds. Then, Lemma
\ref{lemma2} gives that each $\wp _{i}^{e_{i}}$ divides $\mathcal{D}_{L}.$ Since $\gcd
(\wp _{i}^{e_{i}},\wp _{j}^{e_{j}})=\mathcal{O}_{L}$ for all $1\leq i\neq
j\leq r,$ we see that $p\mathcal{O}_{L}\mathcal{=}\wp _{1}^{e_{1}}...\wp
_{s}^{e_{s}}$ divides $\mathcal{D}_{L},$ i. e., $\mathcal{D}_{L}\subset p%
\mathcal{O}_{L},$ and it follows, by the second assertion in Lemma \ref{lemma1}, that $%
t_{L}\mathcal{O}_{L}\subset p\mathcal{O}_{L}.$ Therefore, $t_{L}=p\eta $ for
some $\eta \in \mathcal{O}_{L},$ so that $\eta =t/p\in \mathcal{O}_{L}\cap 
\mathbb{Q}=\mathbb{Z},$ and as $p$ divides $t_{L},$ we have $t_{L}\geq p\geq
2$ and hence $(i)$ is true.$\bigskip $%

\textbf{Proof of Theorem 2. }Suppose $t_{L}\geq 2.$ Then, from the proof of
the implication$\ (i)\Rightarrow (ii)$ in Theorem \ref{theorem1} we see that there is a
prime number $p,$ dividing $d,$ such that $\mathcal{D}_{L}\subset p\mathcal{O%
}_{L},$ and so the norm $p^{d}$ of the principal ideal $p\mathcal{O}_{L}$
divides the norm $\left\vert \text{disc}(L)\right\vert $ of $\mathcal{D}%
_{L};$ thus the relation $(A)\Rightarrow (B)$ is true. Furthermore, the implication $(B)\Rightarrow (C)$ is true\ thanks to Lemma \ref{lemma3}.

Now, assume that $L$ is wild and normal, and let $\wp _{1}^{e_{1}}...\wp
_{s}^{e_{s}}$ be the prime factorization of the ideal $p\mathcal{O}_{L}$ in $%
\mathcal{O}_{L},$ where $p$ is a wildly ramified prime number in $L.$ Since,
the Galois group of the extension $L/\mathbb{Q}$ acts transitively on the
prime ideals $\wp _{1},...,\wp _{s},$ we see that $e_{1}=...=e_{s}.$
Therefore, $p$ divides all powers $e_{i},$ $\forall i\in \{1,...,s\},$ and
we obtain $t_{L}\geq 2$ from Theorem \ref{theorem1}.

Similarly, if $(B)$ holds and $d$ is prime, then $p=d$ \ and Lemma \ref{lemma3} gives
that $p$ ramifies wildly in $L.$ As above, if $\wp _{1}^{e_{1}}...\wp
_{s}^{e_{s}}$ denotes the prime factorization of the ideal $p\mathcal{O}_{L}$
in $\mathcal{O}_{L},$ where $p$ divides some $e_{i},$ say $e_{1}.$ Then, $%
e_{1}\geq p$ and $e_{1}f_{1}+\cdot \cdot \cdot +e_{s}f_{s}=p,$ where $%
p^{f_{1}},...,p^{f_{s}}$ denote the norms of the prime ideals $\wp
_{1},...,\wp _{s},$ respectively. But this forces $s=f_{1}=1,$ $e_{1}=p,$ $p%
\mathcal{O}_{L}\mathcal{=}\wp _{1}^{p},$ i. e., $p$ ramifies totally in $%
\mathcal{O}_{L},$ and so we obtain that the assertion $(A)$ is true thanks
to Theorem \ref{theorem1}.

Finally, if $(B)$ holds and $d=4,$ then $p=2$\ and as before we have that
that $p$ is wildly ramified and $e_{1}f_{1}+\cdot \cdot \cdot +e_{s}f_{s}=4,$
so that $s\leq 3$ and the prime factorization of the ideal $2\mathcal{O}_{L}$
has one of the following shapes: $\wp _{1}^{2},$ $\wp _{1}^{4},$ $\wp
_{1}^{2}\wp _{2}^{2},$ $\wp _{1}^{2}\wp _{2},$ $\wp _{1}^{2}\wp _{2}\wp
_{3}. $ Also, the implication $(iii)\Rightarrow (i)$ in Theorem \ref{theorem1} implies
that $(A) $ is true when $2\mathcal{O}_{L}\in \{\wp _{1}^{2},$ $\wp
_{1}^{4}, $ $\wp _{1}^{2}\wp _{2}^{2}\},$ and in order to conclude it enough
to show that $2\mathcal{O}_{L}\notin \{\wp _{1}^{2}\wp _{2},\wp _{1}^{2}\wp
_{2}\wp _{3}\}.$ Assume that $2\mathcal{O}_{L}=$ $\wp _{1}^{2}\wp _{2}.$
Then, $f_{1}=e_{2}=1, $ $f_{2}=e_{1}=2,$ and Lemma \ref{lemma2} says that $\wp _{2}$
does not appear in the prime factorization of $\mathcal{D}_{L}$ and there is
natural number $n\leq (2-1)+2 v_{2}(2)$ $=3$ such that $\wp _{1}^{n}$
divides $\mathcal{D}_{L}. $ Hence, $v_{2}(\text{disc}(L))\leq 3$ and this
contradicts the fact that $2^{4}$ divides $\text{disc}(L).$ The same
arguments lead to a contradiction when $2\mathcal{O}_{L}=$ $\wp _{1}^{2}\wp
_{2}\wp _{3},$ since in this case $f_{1}=f_{2}=f_{3}=e_{2}=e_{3}=1$ and $%
e_{1}=2.$%
\endproof%

\bigskip

\textbf{Remark.} By the same way as in the proof of the last assertion in
Theorem \ref{theorem2}, we have that $t_{L}\geq 2$ when $L$ is wild and the wildly
ramified prime number $p$ is uniformly ramified, i. e., if $\wp
_{1}^{e_{1}}...\wp _{s}^{e_{s}}$ is the prime factorization of the ideal $p%
\mathcal{O}_{L}$ in $\mathcal{O}_{L},$ then $e_{1}=...=e_{s}.$ The
definition of uniformly ramified prime is inspired by the behavior of
ramification in normal fields (see for instance \cite{mantilla2020introduction}).

\section{Proof of Theorem 3}

\textbf{(I)} We use induction on $n\in \{1,...,r\}.$ There is no thing to
prove when $r=1.$ Suppose $r\geq 2$ and Theorem \ref{theorem3}.(I) is \ true for every $%
n\in \{1,...,r-1\}.$ Let $K$ be the compositum of $L_{1},...,L_{r-1}.$ The
induction hypothesis gives that $\deg (K)=\prod\limits_{1\leq i\leq
r-1}\deg (L_{i}),$\textit{\ }and $t_{K}=1.$

Let $L:=L_{r}K.$ Then, $\deg (K)$ and $\deg (L_{r})$ are factors of $\deg
(L),$ and $\deg (K)\deg (L_{r})$ divides $\deg (L),$ since 
\[
(\gcd (\deg (L_{i}),\deg (L_{r}))=1,\mathit{\ }\forall i\in
\{1,...,r-1\})\Rightarrow \gcd (\deg (K),\deg (L_{r}))=1. 
\]
On the other hand, if $\theta $ is a primitive element of $L_{r}$ over $%
\mathbb{Q},$ then $L_{r}K=K(\theta ),$ $L_{r}=\mathbb{Q}(\theta ),$ and so $%
[L_{r}K:K]\leq \lbrack L_{r}:\mathbb{Q}],$ as the conjugates of $\theta $
over $K$ are among the conjugates of $\theta $ over $\mathbb{Q}.$

It follows by the relations 
\[
\deg (L)=[L_{r}K:K][K:\mathbb{Q}]\leq \lbrack L_{r}:\mathbb{Q}][K:\mathbb{Q}%
]=\deg (L_{r})\deg (K), 
\]%
that $\deg (L)=\deg (L_{r})\deg (K),$ i. e., $\deg (L)=\prod\limits_{1\leq
i\leq r}\deg (L_{i}),$ 
\begin{equation}\label{eq4}
\lbrack L:L_{r}]=[L:\mathbb{Q}]/[L_{r}:\mathbb{Q}]=\deg (K)\text{\ and\ }%
[L:K]=[L_{r}K:K]=\deg (L_{r}).  \tag{4}
\end{equation}%
Now, fix an element $\alpha $ of $\mathcal{O}_{K}$ (resp. an element $\beta $
of $\mathcal{O}_{L_{r}}$) \ such that $T_{K/\mathbb{Q}}(\alpha )=1$ (resp.
such that $T_{L_{r}/\mathbb{Q}}(\beta )=1).$ Then, we have, by \eqref{eq1} and \eqref{eq4}, 
\[
T_{L/\mathbb{Q}}(\alpha )=T_{K/\mathbb{Q}}(T_{L/K}(\alpha ))=T_{K/\mathbb{Q}%
}(\deg (L_{r})\alpha )=\deg (L_{r})T_{K/\mathbb{Q}}(\alpha )=\deg (L_{r}) 
\]%
(resp.

\[
T_{L/\mathbb{Q}}(\beta )=T_{L_{r}/\mathbb{Q}}(T_{L/L_{r}}(\beta ))=T_{L_{r}/%
\mathbb{Q}}(\deg (K)\beta )=\deg (K)T_{L_{r}/\mathbb{Q}}(\beta )=\deg (K)). 
\]%
To conclude it suffices to consider the algebraic integer 
\[
u\alpha +v\beta \in \mathcal{O}_{L}, 
\]%
where the pair $(u,v)\in \mathbb{Z}^{2}$ satisfies $u\deg (L_{r})+v\deg
(K)=1,$ and so we obtain 
\[
T_{L/\mathbb{Q}}(u\alpha +v\beta )=uT_{L/\mathbb{Q}}(\alpha )+vT_{L/\mathbb{Q%
}}(\beta )=u\deg (L_{r})+v\deg (K)=1. 
\]

\textbf{(II)} Let $L=\mathbb{Q}(\sqrt{m_{1}},...,\sqrt{m_{r}}).$ By
rearranging if the necessary the integers $m_{1},...,m_{r},$ set $L=\mathbb{Q%
}(\sqrt{m_{1}},...,\sqrt{m_{s}}),$ where $s$ is minimal. Then, 

\begin{equation}\label{eq5}
\sqrt{m_{i}}\notin \mathbb{Q}(\sqrt{m_{1}},...,\sqrt{m_{i-1}}),  \tag{5}
\end{equation}%

and $[\mathbb{Q}(\sqrt{m_{1}},...,\sqrt{m_{i}}):\mathbb{Q}(\sqrt{m_{1}},...,%
\sqrt{m_{i-1}})]=2,$ $\forall $ $i\in \{2,...,s\};$ thus $\deg (L)=[\mathbb{Q%
}(\sqrt{m_{1}}):\mathbb{Q}]\prod\limits_{2\leq i\leq s}[\mathbb{Q}(\sqrt{%
m_{1}},...,\sqrt{m_{i}}):\mathbb{Q}(\sqrt{m_{1}},...,\sqrt{m_{i-1}})]=2^{s}.$

We also have that $(1+\sqrt{m_{i}})/2\in \mathcal{O}_{\mathbb{Q}(\sqrt{m_{i}}%
)}\subset \mathcal{O}_{L}$ $\ $and the conjugates (over $\mathbb{Q}$) of $(1+%
\sqrt{m_{i}})/2$ are $(1\pm \sqrt{m_{i}})/2,$ $\forall $ $i\in \{1,...,s\}.$
\ Let 
\[
\alpha :=\prod\limits_{1\leq i\leq s}\frac{1+\sqrt{m_{i}}}{2}=\frac{%
1+\sum_{1\leq i\leq s}\sqrt{m_{i}}+\sum_{1\leq i<j\leq s}\sqrt{m_{i}m_{j}}%
+\cdot \cdot \cdot +\sqrt{m_{1}...m_{s}}}{2^{s}}. 
\]%
Then, $\alpha \in \mathcal{O}_{L}$ and 
\[
T_{L/\mathbb{Q}}(\alpha )=\frac{2^{s}+\sum_{1\leq i\leq s}T_{L/\mathbb{Q}}(%
\sqrt{m_{i}})+\sum_{1\leq i<j\leq s}T_{L/\mathbb{Q}}(\sqrt{m_{i}m_{j}}%
)+\cdot \cdot \cdot +T_{L/\mathbb{Q}}(\sqrt{m_{1}...m_{s}})}{2^{s}}. 
\]

In order to obtain the desired result it is enough to show that $T_{L/%
\mathbb{Q}}(\alpha )=1,$ or more precisely $T_{L/\mathbb{Q}}(\beta )=0$ for
any algebraic integer $\beta $ of the form $\sqrt{m_{i_{1}}...m_{i_{j}}},$
where $1\leq i_{1}<\cdot \cdot \cdot <i_{j}\leq s$ \ and $j\in \{1,...,s\}.$
Clearly, $\beta ^{2}\in \mathbb{Z}\Rightarrow \deg (\beta )\leq 2,$ and from 
\eqref{eq5} we get $\beta \notin \mathbb{Z},$ since otherwise there would be $%
j\geq 2$ such that $\sqrt{m_{i_{j}}}\in \mathbb{Q}(\sqrt{%
m_{i_{1}}...m_{i_{j-1}}})\subset \mathbb{Q}(\sqrt{m_{1}},...,\sqrt{%
m_{i_{j}-1}}).$ Hence, $\deg (\beta )=2,$ the conjugates of $\beta $ are $%
\pm \beta ,$ and $T_{L/\mathbb{Q}}(\beta )=2^{s-1}\beta +2^{s-1}(-\beta )=0.$%

\bigskip
\textbf{(III)} By Theorem \ref{theorem2}, a normal field $L$ satisfies $t_L=1$ if and only if $L$ is tame. In order to prove our statement, it suffices to prove that the normal field $KL$ is tame. 

Let $p$ be a ramified prime number in $KL$ and $\mathfrak{b}$ a prime in $\mathcal{O}_{KL}$ dividing $p$; let $\mathfrak{p}:= \mathfrak{b}\cap \mathcal{O}_L$ and $\mathfrak{q}:=\mathfrak{b}\cap \mathcal{O}_K$. Let $G_{\mathfrak{b}}^{KL}$ be the decomposition group of $\mathfrak{b}$ in the Galois group of $KL/L$ and $G_{\mathfrak{q}}^K$ the decomposition group of $\mathfrak{q}$ in $K/\mathbb{Q}$. There is an injective restriction homomorphism $G_{\mathfrak{b}}^{KL} \to G_{\mathfrak{q}}^K$ which restricts to an injective homomorphism between the corresponding inertia subgroups: in particular, the ramification index $e(\mathfrak{b}|\mathfrak{p})$ divides the ramification index $e(\mathfrak{q}|p)$, and since the last one is not a multiple of $p$, so is the first one. Now, by multiplicativity we have $e(\mathfrak{b}|p) = e(\mathfrak{b}|\mathfrak{p})\cdot e(\mathfrak{p}|p)$ and both the factors are not multiples of $p$; since every other factor of $p$ in $KL$ has the same ramification index thanks to the normality of $KL$, we have that $p$ is tamely ramified, and doing this for every ramified prime we obtain that $KL$ is tame.\\

\end{document}